\documentclass[10pt,leqno]{amsart}
\usepackage{amsmath,amsfonts,amssymb,latexsym}
\usepackage{mathtools}
\usepackage{enumerate,comment}
\mathtoolsset{showonlyrefs=true}
\usepackage[pagewise]{lineno}
\usepackage[dvipdfmx]{graphicx}
\numberwithin{equation}{section}

\newtheorem{theorem}{Theorem}[section]
\newtheorem{corollary}[theorem]{Corollary}

\newtheorem{proposition}[theorem]{Proposition}

\theoremstyle{definition}

\newtheorem{remark}[theorem]{Remark}

\numberwithin{equation}{section}

\usepackage{mathtools}
\mathtoolsset{showonlyrefs=true}

\numberwithin{equation}{section}

\newcommand{\BR}{\mathbb{R}}
\newcommand\tU{\widetilde{U}}
\newcommand\tV{\widetilde{V}}
\newcommand{\pd}{\partial}

\begin{document}

\title[Speed of traveling waves for diffusive Lotka-Volterra competition system]
{Propagation speed of traveling waves for diffusive Lotka-Volterra system with strong competition}

\author[K.-I. Nakamura]{Ken-Ichi Nakamura}
\address{Meiji Institute for Advanced Study of Mathematical Sciences, Meiji University, Tokyo, Japan}
\email{kenichi\_nakamura@meiji.ac.jp}

\author[T. Ogiwara]{Toshiko Ogiwara}
\address{Department of Mathematics, Josai University, Saitama, Japan}
\email{toshiko@josai.ac.jp}

\medskip

\thanks{Date: \today. Corresponding Author: Ken-Ichi Nakamura}

\thanks{This work was supported by JSPS KAKENHI Grant Number JP21K03368, JP21KK0044, 
JP22K03418 and JP25K07124. The first author was partially supported by MEXT Promotion of Distinctive Joint Usage/Research Center Support Program Grant Number JPMXP0724020292.}




\begin{abstract}
We study the propagation speed of bistable traveling waves in the classical 
two-component diffusive Lotka-Volterra system under strong competition. 
From an ecological perspective, the sign of the propagation speed determines 
the long-term outcome of competition between two species and thus plays a central role in predicting the success or failure of invasion of an alien species into 
habitats occupied by a native species. Using comparison arguments, we establish 
sufficient conditions determining the sign of the propagation speed, which refine  
previously known results. In particular, we show that in the symmetric case, 
where the two species differ only in their diffusion rates, 
the faster diffuser prevails over a substantially broader parameter range 
than previously established. Moreover, we demonstrate that when the interspecific 
competition coefficients differ significantly, the outcome of competition 
cannot be reversed by adjusting diffusion or growth rates. These findings 
provide a rigorous theoretical framework for analyzing invasion dynamics, 
offering sharper mathematical criteria for invasion success or failure.
\end{abstract}

\maketitle

\section{Introduction}\label{sec1}

In ecology, a central research theme is to understand whether invasive alien species can successfully invade into habitats already occupied by native species (see, for example, \cite{SK97,LPP16}). The diffusive Lotka-Volterra competition system
\begin{equation}\label{LV2}
\begin{cases}
U_t=U_{xx}+U(1-U-k_1V), \quad & x\in\BR, \ t>0, \\
V_t=dV_{xx}+rV(1-k_2U-V), & x\in\BR, \ t>0, 
\end{cases}
\end{equation}
is a classical model frequently employed to describe the spatio-temporal dynamics of such invasions. This system characterizes the time evolution of the population densities of two dispersing species 
competing for the same resource. Here $U(x,t)$ and $V(x,t)$ denote the normalized 
population densities of the species at location $x$ and time $t$, with carrying 
capacities normalized to $1$. The parameters are positive 
constants: $d$ represents the ratio of diffusion coefficients, 
$r$ the ratio of net growth rates, and $k_1, k_2$ the interspecific competition coefficients. 

Throughout this paper, we assume the strong competition condition (or bistable 
condition)
\begin{equation}\label{SC}
k_1>1, \enskip k_2>1, 
\end{equation}
which indicates that interspecific competition is stronger than intraspecific competition for both species. Under this assumption, the system admits two stable 
constant equilibria $(0,1)$ and $(1,0)$, as well as unstable constant equilibria 
$(0,0)$ and  
\[ 
(U_*,V_*)=\left(\frac{k_1-1}{k_1k_2-1}, \frac{k_2-1}{k_1k_2-1}\right).
\]

The success or failure of invasion can be mathematically 
characterized by the existence and qualitative properties 
of bistable traveling wave solution $(U(x,t),V(x,t))=(\Phi(x+ct),\Psi(x+ct))$ 
of \eqref{LV2} 
connecting two stable equilibria $(1,0)$ and $(0,1)$. 
Here $(\Phi(z),\Psi(z))$ and the propagation speed $c$ 
satisfy 
\begin{equation}\label{profileLV}
\begin{cases}
\Phi''-c\Phi'+\Phi(1-\Phi-k_1\Psi)=0, & z\in\BR, \\
d\Psi''-c\Psi'+r\Psi(1-k_2\Phi-\Psi)=0, & z\in\BR, \\
(\Phi(-\infty),\Psi(-\infty))=(0,1), \ 
(\Phi(\infty),\Psi(\infty))=(1,0).
\end{cases}
\end{equation}
The existence of bistable traveling waves for competition-diffusion systems 
including \eqref{LV2} has been studied in 
\cite{Ga82,CG84,Ka95}; see also \cite{VVV94,FZ15} for extensions to 
more general monotone systems. In particular, 
under the strong competition condition \eqref{SC}, system \eqref{LV2} admits 
a unique (up to translation) monotone traveling wave solution, which is stable 
in an appropriate functional setting. Moreover, the propagation speed $c=c(d,r,k_1,k_2)$ is uniquely determined by the system parameters (see \cite{Ga82,Ka95,KF96}).
Here the monotonicity of 
the traveling wave means that $\Phi'(z)>0>\Psi'(z)$ for all $z\in\BR$. 

From an ecological perspective, 
the sign of $c$ is a key factor that determines which species ultimately 
dominates: if $U$ represents the native 
species and $V$ the alien species, then $c<0$ implies successful 
invasion by $V$, whereas $c>0$ indicates that the invasion fails and $U$ 
ultimately prevails. 
From a mathematical perspective, determining the sign of $c$ provides the theoretical foundation for analyzing invasion dynamics and clarifying 
the role of spatial dispersal in shaping competitive outcomes.  
For instance, Carr\`ere \cite{Ca18} and Peng, Wu and Zhou \cite{PWZ21} 
demonstrated that the sign of $c$ crucially influences the asymptotic behavior 
of solutions of \eqref{LV2} under the strong competition condition \eqref{SC}, 
thereby providing rigorous justification for 
interpreting traveling waves as reliable predictors of invasion success or failure. 

Alternatively, the following two-component system can be employed as a model 
to describe the same phenomenon: 
\begin{equation}\label{LV1}
\begin{cases}
\tU_t=\tU_{xx}+\tU(1-\tU-\gamma \tV), \quad & x\in\BR, \ t>0, \\
\tV_t=d\tV_{xx}+\tV(\alpha-\beta \tU-\tV), & x\in\BR, \ t>0. 
\end{cases}
\end{equation}
The strong competition condition for \eqref{LV1} is given by 
\begin{equation}\label{SC1}
\dfrac{1}{\gamma}<\alpha<\beta.
\end{equation}
Problems \eqref{LV1} with \eqref{SC1} is equivalent to \eqref{LV2} with \eqref{SC} 
under the correspondence $(\tU,\tV)=(U,\alpha V)$ and $\alpha=r$, $\beta=rk_2$, 
$\gamma=k_1/r$. Consequently, 
$(\tU(x,t),\tV(x,t))=(\Phi(x+ct),\alpha\Psi(x+ct))$ is a unique traveling wave 
solution of \eqref{LV1}, where $(\Phi,\Psi)$ and $c$ are as in \eqref{profileLV}. 
The speed $c$ remains unchanged and is uniquely determined 
by the parameters $d, \alpha, \beta$ and $\gamma$. Kan-on \cite{Ka95} proved that 
for any $d>0$ and any $\beta, \gamma>0$ with $\beta\gamma>1$, there exists a 
unique value $\alpha_*=\alpha_*(d,\beta,\gamma)\in(1/\gamma, \beta)$ such that
$c(d,\alpha_*,\beta,\gamma)=0$. He further established the following monotonic 
dependence of $c$ on the parameters $\alpha, \beta, \gamma$: 
\begin{equation}\label{monotone1}
\dfrac{\pd c}{\pd \alpha}<0, \enskip \dfrac{\pd c}{\pd \beta}>0, \enskip 
\dfrac{\pd c}{\pd \gamma}<0
\end{equation}
for $d>0$ and $\alpha, \beta, \gamma$ satisfying \eqref{SC1}. Therefore, for any 
$d>0$ and for any $\alpha,\beta,\gamma$ satisfying \eqref{SC1}, we obtain 
\[
c\lesseqqgtr0 \ \Longleftrightarrow \ \alpha\gtreqqless \alpha_*(d,\beta,\gamma).
\]
However, determining the exact value of $\alpha_*$ for given $d, \beta,\gamma$ is 
generally difficult, except in special cases where additional parameter relations 
hold, as in \cite{RM2001}. 

Concerning the propagation speed $c=c(d,r,k_1,k_2)$ for the traveling wave 
of \eqref{LV2}, it follows from \eqref{monotone1} that 
\begin{equation}\label{monotone2}
\dfrac{\pd c}{\pd k_1}<0<\dfrac{\pd c}{\pd k_2}. 
\end{equation}
In contrast, the monotone dependence of $c$ on the parameters $d$ and 
$r$ remains unknown. 
Recently, Xiao \cite[Theorem 1.1]{Xi2025} proved that for any $d, r>0$ and $k_2>1$, 
there exist constants $k_->k_+>1$ such that $c<0$ whenever $k_1\geq k_-$, 
and $c>0$ whenever $1<k_1<k_+$. Consequently, there exists a threshold value 
$k_*=k_*(d,r,k_2)\in(k_+,k_-)$ satisfying 
\[
c\lesseqqgtr0 \ \Longleftrightarrow \ k_1\gtreqqless k_*(d,\beta,\gamma).
\]
However, the proof of the existence of $k_\pm$ relies on limiting arguments as $
k_1\to\infty$ and $k_1\to1+$, and no quantitative estimate of these values is 
provided in \cite{Xi2025}. There are also several studies on the 
propagation speed $c$ (see, \cite{GL2013,GN2015,MHO2019,MZYO20,CCW2023,MNO2023}), 
nevertheless identifying explicit parameter conditions that determine the sign 
of $c$ remains a challenging mathematical problem.

Building on these observations, the aim of the present paper is to 
significantly refine the parameter ranges for which the sign of the propagation 
speed $c$ can be determined. Addressing this problem is crucial for linking ecological interpretation with rigorous mathematical results, and it constitutes 
the main focus of the present study. Our approach relies on the construction of 
time-independent supersolutions for blocking wave propagation, which enables us 
to derive explicit conditions 
ensuring $c<0$. As a consequence, we obtain sharp criteria for invasion success 
and substantially extend the parameter regimes in which the sign of $c$ is fully 
characterized. Particular emphasis is placed on the symmetric nonlinearity case 
($r=1$ and $k_1=k_2$), where our results considerably improve previously known 
results, including those summarized in the review by Girardin \cite{Gi2019}. 

The paper is organized as follows: In Section \ref{S:super}, 
we transform \eqref{LV2} into a cooperative system and construct a time-independent 
supersolution. This supersolution blocks the leftward propagation of traveling 
waves, thereby showing that the propagation speed $c$ is nonpositive. Based on this 
construction, we derive in Section \ref{S:sign} sufficient conditions on the 
parameters $(d, r, k_1, k_2)$ for $c$ to be negative (Theorem \ref{thm:negative1}).
By exchanging the roles of the two species, we also obtain conditions ensuring 
positive speed. 

Section \ref{S:symmetric} is devoted to the symmetric case $(r=1, k_1=k_2=:k>1)$, where the two species differ only in their diffusion rates. As in \cite{Gi2019}, 
numerical evidence suggests that the faster diffuser always prevails 
(i.e., $c<0$ if $d>1$ and $k>1$), while rigorous results have 
so far been established only for a limited range of parameters. By applying 
Theorem \ref{thm:negative1}, we considerably enlarge the parameter region $(d, k)$ 
for which the speed $c$ is proved to be negative (Theorem \ref{thm:negative-sym}). 

Section \ref{S:degenerate} establishes sufficient conditions for determining the sign of $c$ when the diffusion ratio $d$ is small. Section \ref{S:determinacy} 
shows that when the interspecific competition coefficients $k_1$ and $k_2$ differ 
greatly, the sign of $c$ remains unchanged for all $d, r>0$ 
(Theorem \ref{thm:determinacy}). This indicates that 
if the competitive strengths of the two species are highly asymmetric, the 
outcome of invasion cannot be altered by adjusting $d$ and $r$. 
From an ecological perspective, this is a particularly significant and intriguing 
finding. 

%%%%%%%%%%%%%%%%
\section{Construction of time-independent supersolutions}\label{S:super}
By the transformation $(u,v)=(U,1-V)$, the system \eqref{LV2}
can be rewritten as the following cooperative system: 
\begin{equation}\label{Co1}
\begin{cases}
u_t=u_{xx}+f(u,v), \quad & x\in\BR, \ t>0, \\
v_t=dv_{xx}+rg(u,v), & x\in\BR, \ t>0, 
\end{cases}
\end{equation}
where 
\begin{equation}\label{def:fg}
f(u,v):=u(1-u-k_1(1-v)), \quad g(u,v):=(1-v)(k_2u-v).
\end{equation}
The system \eqref{Co1} possesses two stable 
constant equilibria $(0,0), (1,1)$, together with two unstable equilibria $(0,1)$, 
$(u_*, v_*)$, where 
$$u_*=\frac{k_1-1}{k_1k_2-1}, \enskip v_*=\frac{k_2(k_1-1)}{k_1k_2-1}.$$ 
%Note that $0<u_*<v_*<1$ by the strong competition condition \eqref{SC}. 
Since 
\[
\frac{\pd f}{\pd v}\geq 0, \ \frac{\pd g}{\pd u}\geq 0 \quad \mbox{in} \ R:=\{(u,v) \mid u\geq0, \ v\leq 1\},
\]
the comparison theorem is valid for supersolutions and subsolutions of \eqref{Co1} lying in $R$. 

The unique traveling wave (up to translation) of \eqref{Co1} 
connecting $(0,0)$ and 
$(1,1)$ is given by $(\phi(x+ct), \psi(x+ct))$ with $\phi=\Phi$, $\psi=1-\Psi$ 
and with the same propagation speed $c$, where $(\Phi,\Psi)$ and $c$ are as in \eqref{profileLV}. Equivalently, $(\phi(z), \psi(z))$ 
and $c$ satisfy
\begin{equation}\label{profileCoop}
\begin{cases}
\phi''-c\phi'+f(\phi,\psi)=0, & z\in\BR, \\
d\psi''-c\psi'+rg(\phi,\psi)=0, & z\in\BR, \\
(\phi(-\infty),\psi(-\infty))=(0,0), & (\phi(\infty),\psi(\infty))=(1,1).
\end{cases}
\end{equation}

In this section, we will construct a time-independent supersolution 
$(\phi_+(x),\psi_+(x))$ of \eqref{Co1} satisfying  
$(\phi_+(-\infty),\psi_+(-\infty))=(0,0)$ and 
$(\phi_+(\infty),\psi_+(\infty))=(1,1)$ by employing a variant of sigmoidal 
functions. 
This supersolution blocks the leftward propagation of the traveling wave 
$(\phi(x+ct),\psi(x+ct))$, and consequently, we conclude that the 
propagation speed $c$ is nonpositive. 

For $p>1$, let $h_p\in C^1(\BR)$ be defined by 
\begin{equation}\label{def:h}
h_p(s):=
\begin{cases}
s(1-s)(s^{p-1}-\alpha_p), \quad & s\geq0, \\
-\alpha_p s, & s<0,
\end{cases}
\end{equation}
where 
\[
\alpha_p:=\frac{6}{(p+1)(p+2)}\in(0,1).
\]
Then $h_p$ is of bistable type with three zeroes $0, \alpha_p^{1/(p-1)}, 1$ and has 
the balanced property
\begin{equation}\label{balance}
\int_0^1 h_p(s)ds=0.
\end{equation}
It is known (see, for example, \cite{K62} and \cite{FM77}) that \eqref{balance} guarantees the existence of 
a strictly monotone increasing function $\sigma=\sigma_p(x)$ solving 
\begin{equation}\label{eq:sigmoid}
\begin{cases}
\sigma''+h_p(\sigma)=0, \quad x\in\BR, \\
\sigma(-\infty)=0, \ \sigma(\infty)=1. 
\end{cases}
\end{equation}
Note that for $p=2$, $h_2(s)=s(1-s)(s-1/2)$ and thus 
$\sigma_2$ is a sigmoidal function given by $\sigma_2(x)=(1+e^{-x/\sqrt{2}})^{-1}$.

\begin{proposition}\label{prop:supersol}
Set $(\phi_+(x),\psi_+(x))=(\sigma_p(ax)^p,\sigma_p(ax))$ for $a>0$. Then, $(\phi_+,\psi_+)$ is a time-independent 
supersolution of \eqref{Co1} if all the following conditions hold: 
\begin{itemize}
\item[\rm (a)] $a^2<\dfrac{(p+1)(p+2)}{6p^2}(k_1-1)$;
\item[\rm (b)] Either $p\leq k_1$ or $\left(p>k_1 \mbox{ and } \ a^2\geq \dfrac{(p+1)(p+2)(p-k_1)}{p(p-1)(p+4)}\right)$;
\item[\rm (c)] $p<2k_1$ and $a^2\leq \dfrac{2k_1-p}{2p}$;
\item[\rm (d)] $\dfrac{r(k_2-1)}{d}\dfrac{(p+1)(p+2)}{(p-1)(p+4)}\leq a^2\leq \dfrac{r(p+1)(p+2)}{6d}$.
\end{itemize}
\end{proposition}
\begin{proof} Assuming the conditions (a)-(d), we will show that the functions 
\begin{equation}\label{def:IJ}
I(x):=\phi_+''(x)+f(\phi_+(x),\psi_+(x)), \enskip J(x):=\dfrac{d}{r}\psi_+''(x)+g(\phi_+(x),\psi_+(x))
\end{equation}
are both nonpositive for $x\in\BR$. First we note that by \eqref{def:h} and \eqref{eq:sigmoid}, $\sigma=\sigma_p$ satisfies
\begin{align}
\sigma''&= -h_p(\sigma)=\alpha_p\sigma(1-\sigma)-\sigma^p(1-\sigma), \\
\left(\sigma'\right)^2&=-2\int_0^\sigma h_p(s)ds=\alpha_p\sigma^2-\dfrac{2}{3}\alpha_p\sigma^3-\frac{2}{p+1}\sigma^{p+1}+\frac{2}{p+2}\sigma^{p+2}.
\end{align}
Hence we have $J(x)=s(1-s)J_1(s)$, where $s=\sigma_p(ax)\in(0,1)$ and 
\begin{equation}
J_1(s)=\frac{d}{r}a^2\alpha_p-1+\left(k_2-\frac{d}{r}a^2\right)s^{p-1}.
\end{equation}
Since $J_1$ is monotone in $s\in(0,1)$, $J_1\leq\max\{J_1(0),J_1(1)\}\leq0$. Here the last inequality follows from (d). 

Similarly, direct 
calculation yields $I(x)=s^p(A+Bs+Cs^{p-1}+Ds^p)$, where $s=\sigma_p(ax)$ and \begin{alignat}{2}
A&=\frac{6p^2}{(p+1)(p+2)}a^2-(k_1-1), & \quad B&=-\frac{2p(2p+1)}{(p+1)(p+2)}a^2+k_1, \\
C&=-\frac{p(3p-1)}{p+1}a^2, & D&=\frac{3p^2}{p+2}a^2-1.
\end{alignat}
Then, $A+B+C+D=0$ and hence $I=\tau^{-2p}(\tau-1)I_1(\tau)$, where 
$\tau=s^{-1}>1$ and 
\[
I_1(\tau)=A\dfrac{\tau^p-1}{\tau-1}+B\dfrac{\tau^{p-1}-1}{\tau-1}+C.
\] 
Now we show that $I_1(\tau)<0$ for $\tau>1$ if (a), (b) and 
(c) are satisfied. We remark that the conditions (a), (b), (c) are equivalent to 
\[
\mbox{(A)} \enskip A<0, \quad \mbox{(B)} \enskip pA+(p-1)B+C\leq0, 
\quad \mbox{(C)} \enskip pA+(p-2)B\leq0,
\]
respectively. By (A) and (B), we have $I_1(1+)=pA+(p-1)B+C\leq0$ and 
$I_1(\infty)=-\infty$. Set 
\[
I_2(\tau):=(\tau-1)^2I_1'(\tau)
=(p-1)A\tau^p+\{-pA+(p-2)B\}\tau^{p-1}-(p-1)B\tau^{p-2}+A+B.
\]
Then, the conditions (A) and (C) imply that $I_2(1+)=0$, $I_2(\infty)=-\infty$ and 
\[
I_2'(\tau)=(p-1)\tau^{p-3}(\tau-1)\{pA\tau+(p-2)B\}< 0,  
\]
for $\tau>1$. 
Thus we obtain $I_1'(\tau)<0$ for $\tau>1$ and hence $I_1(\tau)<I_1(1+)\leq0$. The proposition is proved. 
\end{proof}

\begin{corollary}\label{cor:negative}
If there exist constants $p>1$ and $a>0$ satisfying conditions 
{\rm (a)-(d)} in Proposition {\rm \ref{prop:supersol}}, then $c\leq0$. 
\end{corollary}
\begin{proof} 
We only outline the proof since it relies on a standard comparison argument. 

Let $(\phi(x+ct),\psi(x+ct))$ be a traveling wave of \eqref{Co1}. 
By Lemma A2 in \cite{OM99} (see also \cite{XC97,GNOW19,WO21}), 
one can construct a subsolution $(u_-(x,t),v_-(x,t))$ of \eqref{Co1} of the form 
\begin{align*}
u_-(x,t)&=\phi(x+ct-\delta(1-e^{-\nu t}))-\sigma\delta\rho_1(x+ct)e^{-\nu t}, \\
v_-(x,t)&=\psi(x+ct-\delta(1-e^{-\nu t}))-\sigma\delta\rho_2(x+ct)e^{-\nu t},
\end{align*}
where $\rho_1, \rho_2$ are smooth positive bounded functions on $\BR$, and 
$\delta, \nu, \sigma$ are positive constants. Furthermore, the constant $\delta$ 
can be chosen arbitrarily small. 

 Now suppose $c>0$. Let $(\phi_+(x),\psi_+(x))=(\sigma_p(ax)^p,\sigma_p(ax))$ be the 
time-independent supersolution of \eqref{Co1} obtained in Proposition \ref{prop:supersol}. 
We can then take a sufficiently large $x_0\in\BR$ and a sufficiently small 
$\delta>0$ such that  
\[
\phi_+(x+x_0)\geq \max\{u_-(x,0),0\}, \quad \psi_+(x+x_0)\geq \max\{v_-(x,0),0\}
\]
for $x\in\BR$. By the comparison theorem, it follows that 
\[
\phi_+(x+x_0)\geq \max\{u_-(x,t),0\}, \quad \psi_+(x+x_0)\geq \max\{v_-(x,t),0\}
\]
for all $x\in\BR$ and $t\geq0$. Since $c>0$, the right-hand sides converge to $1$ as 
$t\to\infty$, whereas the left-hand sides remain strictly less than $1$ for all 
$x\in\BR$. 
This contradiction completes the proof. 
\end{proof}

\begin{remark} By the uniqueness (up to translation) of the bistable traveling wave for \eqref{LV2} (or \eqref{Co1}), the speed $c=c(d,r,k_1,k_2)$ satisfies
\begin{equation}\label{formula}
c(d,r,k_1,k_2)=-\sqrt{dr}\ c(1/d,1/r,k_2,k_1).
\end{equation}
See \cite[Section 6]{MHO2019} for details. In view of this formula, one also obtains 
sufficient conditions for $c\geq0$ by applying the correspondence 
$(d,r,k_1,k_2)\mapsto(1/d,1/r,k_2,k_1)$ to the conditions (a)-(d).  
\end{remark}
%%%%%%%%%%%%%%%%%%%%%%%%%%%%%%%%%%%%%%%%%%%%%%%%%%%%%%%%%%%%%%%%%%%%%%%%%%%%
\section{Determining the sign of the propagation speed of bistable traveling waves}\label{S:sign}
In this section, we determine the sign of the propagation speed $c$ for the 
bistable traveling wave $(\phi(x+ct),\psi(x+ct))$ 
using Proposition \ref{prop:supersol} and Corollary \ref{cor:negative}. 

For $k\geq1$, we define 
\begin{equation}\label{def:m}
m(k):=\dfrac{\sqrt{24k+1}-3}{2}.
\end{equation}
Note that $m(1)=1$, $m(2)=2$ and that
\begin{equation}\label{property:m}
m(k)
\begin{cases}
>k \quad & \mbox{if} \enskip 1<k<2, \\
<k & \mbox{if} \enskip k>2.
\end{cases}
\end{equation}

\begin{theorem}\label{thm:negative1}
The speed $c$ is negative if either of the following conditions holds:
\makeatletter\tagsleft@true\makeatother
\begin{align}
k_1\geq m(k_2), \quad &
\dfrac{d}{r}>
\begin{cases}
\dfrac{6k_1^{\,2}}{(k_1-1)^2(k_1+4)}(k_2-1) & (k_1<2), \\[7pt]
\dfrac{4}{k_1-1}(k_2-1) & (k_2\leq 2\leq k_1), \\[7pt]
\dfrac{2k_2m(k_2)}{2k_1-m(k_2)} & (k_2>2), 
\end{cases}\label{neg1}\tag{N1} \\
1<k_1<m(k_2), \quad & 
\dfrac{m(k_2)(k_2-1)}{m(k_2)-k_1}>\dfrac{d}{r}>
\begin{cases}
\dfrac{m(k_2)^2}{k_1-1} & (k_2\leq2), \\[7pt]
\dfrac{2k_2m(k_2)}{2k_1-m(k_2)} & (k_2>2).
\end{cases}
\label{neg2}\tag{N2}
\end{align}
\end{theorem}
\begin{proof} First we show that if we assume \eqref{neg1} or \eqref{neg2}, we can find 
$p>1$ and $a>0$ satisfying all the conditions (a)-(d) in Proposition \ref{prop:supersol}. 
In view of Proposition \ref{prop:supersol} (d), the condition $p\geq m(k_2)$ is required. 

In the case of \eqref{neg1}, we can take $p>1$ satisfying 
\begin{equation}\label{ineq:p}
m(k_2)\leq p\leq k_1
\end{equation}
and 
\begin{align}
\dfrac{r(k_2-1)}{d}\dfrac{(p+1)(p+2)}{(p-1)(p+4)}
&<\min\left\{\dfrac{(p+1)(p+2)}{6p^2}(k_1-1),\dfrac{2k_1-p}{2p}\right\} \label{ineq1}\\
&=
\begin{cases}
\dfrac{(p+1)(p+2)}{6p^2}(k_1-1) & (1<p\leq 2), \\[5pt]
\dfrac{2k_1-p}{2p} & (p\geq 2).
\end{cases}\notag
\end{align}
In fact, 
%since $m(k_2)\geq k_2$ for $1<k_2\leq 2$ and $m(k_2)<k_2$ for $k_2>2$, 
taking $p=k_1$ if $k_1<2$, $p=2$ if $m(k_2)\leq 2\leq k_1$ or $p=m(k_2)$ if $m(k_2)>2$, 
we see that \eqref{neg1} implies \eqref{ineq1}. 
Furthermore, by \eqref{ineq:p} and \eqref{ineq1}, we can find $a>0$ such that the conditions (a), (c), 
(d) and the former condition of (b) in Proposition \ref{prop:supersol} hold true. Therefore, Proposition 
\ref{prop:supersol} yields $c\leq0$.

In the case of \eqref{neg2}, we take $p=m(k_2)>k_1$. Then, the condition (d) holds for $a^2=k_2r/d$. Furthermore, \eqref{neg2} implies the condition (a) and the latter conditions of (b) and (c)
in Proposition \ref{prop:supersol}. The condition $p=m(k_2)<2k_1$ in (c) is also satisfied if the 
second condition of \eqref{neg2} holds 
(in other words, if the left-hand side is larger than the right-hand side in the condition). Hence 
we have $c\leq0$. 

Next we show that $c$ is negative if either \eqref{neg1} or \eqref{neg2} is satisfied. Since the speed 
$c=c(d,r,k_1,k_2)$ 
is strictly monotone decreasing in $k_1$, we easily see that $c<0$ except for the case where 
$k_1=m(k_2)$ in \eqref{neg1}. Let $d, r$ and $k_2$ be fixed and suppose that \eqref{neg1} is satisfied for $k_1=m(k_2)$. Then 
we can take sufficiently small $\varepsilon>0$ such that the condition \eqref{neg2} holds for 
$k_1=m(k_2)-\varepsilon$. Hence, the strict monotonicity of $c$ in $k_1$ shows that 
$c(d,r,m(k_2),k_2)<c(d,r,m(k_2)-\varepsilon,k_2)\leq0$.
\end{proof}

\begin{remark}
The choice of $p$ satisfying \eqref{ineq:p} and \eqref{ineq1} in the above proof is numerically 
optimal for minimizing the lower bound of $d/r$ in 
\eqref{neg1}. 
\end{remark}

\begin{corollary}\label{cor:xiao} Let $d,r>0$ and $k_2>1$ be fixed. 
\begin{itemize}
\item[\rm (i)] The speed $c$ is negative if 
\begin{equation}\label{neg3}
k_1>
\begin{cases}
\max\left\{2, 1+\dfrac{4r}{d}(k_2-1)\right\} & (1<k_2\leq2), \\[7pt]
m(k_2)\max\left\{1, \dfrac12+\dfrac{r}{d}k_2\right\} & (k_2>2).
\end{cases}
\end{equation}
\item[\rm (ii)] The speed $c$ is positive if 
\begin{equation}\label{pos1}
0<k_1-1<
\begin{cases}
\dfrac16(k_2-1)(k_2+4)\min\left\{1,\dfrac{r}{d}\dfrac{k_2-1}{k_2^{\,2}}\right\} & (1<k_2\leq2), \\[7pt]
(k_2-1)\min\left\{\dfrac{k_2+4}{6},\dfrac{r}{4d}\right\} & (k_2\geq2).
\end{cases}
\end{equation}
\end{itemize}
\end{corollary}
\begin{proof}
(i) Since the condition \eqref{neg3} implies \eqref{neg1}, $c$ is negative.  

(ii) Let $d,r>0$ and $k_1>1$ be fixed. Then, we see from \eqref{neg1} that $c$ is negative if 
\begin{align}
k_2&\leq \frac16(k_1+1)(k_1+2), \quad k_2-1<\frac{d}{r}\frac{(k_1-1)^2(k_1+4)}{6k_1^{\,2}} & &(1<k_1<2), \\
k_2&\leq \frac16(k_1+1)(k_1+2), \quad k_2-1<\dfrac{d}{4r}(k_1-1) & &(k_1\geq2).
\end{align}
The assertion follows from these conditions and the formula \eqref{formula}.\end{proof}

\begin{remark} As stated in the introduction, there exists a threshold 
$k_*=k_*(d,r,k_2)>1$ with the following property:
\begin{equation}
c\lesseqqgtr0 \quad \mbox{if} \enskip k_1\gtreqqless k_*. 
\end{equation}
Corollary \ref{cor:xiao} 
gives an upper bound and a lower bound of $k_*$. 
\end{remark}

%%%%%%%%%%%%%%%%%%%
\section{Symmetric nonlinearity case}
\label{S:symmetric}
In the special case where $r=1$ and $k_1=k_2=:k>1$, \eqref{LV2} reduces to 
\begin{equation}\label{LV2S}
\begin{cases}
U_t=U_{xx}+U(1-U-kV), \quad & x\in\BR, \ t>0, \\
V_t=dV_{xx}+V(1-V-kU), & x\in\BR, \ t>0. 
\end{cases}
\end{equation}
Here, the two species differ only in their diffusion rates, and thus this symmetric model reduces the invasion problem to determining whether the slower diffuser or 
the faster diffuser will ultimately prevail in the competition. 
In the review paper of Girardin \cite{Gi2019}, this problem --- referred to as 
the ``Unity is strength'' versus ``Disunity is strength'' dichotomy --- was treated and a global ``Disunity is strength''-type result (namely, 
$c<0$ for all $d>1$ and $k>1$) was numerically suggested. However, the problem is far from fully understood. Indeed, several sufficient 
conditions for negative propagation speed was summarized in \cite{Gi2019} 
as follows: 
\begin{itemize}
\item[(i)] Rodrigo and Mimura \cite{RM2001}: $(d,k)=(11/2,11/6)$. 
\item[(ii)] Guo and Lin \cite{GL2013}: $d=4$ and $5/4\leq k\leq 4/3$. 
\item[(iii)] Ma, Huang and Ou \cite{MHO2019}: $5/3<k<2$ and $4<d<4/(k-1)$, $d\not=2k/(k-1)$. 
\item[(iv)] Alzahrani, Davidson and Dodds \cite{ADD10}: $k>1$ and $d>\underline{d}(k)$ 
for sufficiently large $\underline{d}(k)>1$. 
\item[(v)] Girardin and Nadin \cite{GN2015}: $d>1$ and $k>\underline{k}(d)$ for sufficiently large 
$\underline{k}(d)>1$. 
\item[(vi)] Risler \cite{Ri2017}: $d=1+\delta d$ and $k=1+(\delta k)^2$ in the parameter regime $
0<\delta d\ll \delta k\ll 1$. 
\end{itemize}
Note that in the limiting cases (iv), (v) and (vi), no quantitative information 
is available for  $\underline{d}(k)$, $\underline{k}(d)$, $\delta d$ and $\delta k$. 

Recently, additional sufficient conditions for negative speed have been obtained:
\begin{itemize}
\item[(vii)] Chang, Chen and Wang \cite{CCW2023}: 
\begin{equation}
\max\left\{k-\dfrac{d(k-1)}{3k-1}, \dfrac{4d(k-1)}{(3k-1)^2}+
\left\lfloor\dfrac{2d(k+1)}{(3k-1)^2}-k\right\rfloor
\left\lfloor\dfrac{k(5-3k)}{2}\right\rfloor\right\}<1, 
\end{equation}
where $\lfloor\cdot\rfloor$ denotes the floor function. 
\item[(viii)] Morita, Nakamura and Ogiwara \cite{MNO2023}: $5/3<k<2$ and $4<d<2/(2-k)$.
\end{itemize}
Figure \ref{fig:sym1} (left) illustrates the above-mentioned regions 
of negative speed in the $(d,k)$-plane, excluding the limiting cases. 
See also \cite[Figure 2]{Gi2019}, which depicts the regions (i)-(vi). 
Thus, the sign of the propagation speed $c$ remains unknown for a wide range of 
parameter values $(d,k)$. 

By virtue of Theorem \ref{thm:negative1}, we obtain the following 
sufficient conditions for negative speed in \eqref{LV2S}: 

\begin{theorem}\label{thm:negative-sym}
The propagation speed $c$ of the bistable traveling wave for \eqref{LV2S} is negative if either of the following conditions holds: 
\makeatletter\tagsleft@true\makeatother
\begin{align}
& k\geq2, \enskip d>\dfrac{2km(k)}{2k-m(k)}, \label{negS1}\tag{S1}\smallskip \\
& 1<k<2, \enskip \dfrac{m(k)^2}{k-1}<d<\dfrac{m(k)(k-1)}{m(k)-k}, \label{negS2}\tag{S2}
\end{align}
where $m(k)=(\sqrt{24k+1}-3)/2$ as defined in \eqref{def:m}. 
\end{theorem}
\begin{proof}
By \eqref{property:m}, the conditions \eqref{negS1} and \eqref{negS2} imply 
\eqref{neg1} and 
\eqref{neg2}, respectively. Hence, the assertion of the theorem follows from 
Theorem \ref{thm:negative1}.
\end{proof}
Note that the union of the regions \eqref{negS1} and \eqref{negS2} in the $(d,k)$-plane is unbounded 
in both $d$ and $k$. 
As shown in Figure \ref{fig:sym1} (right), our conditions cover a substantially 
larger parameter region in the $(d,k)$-plane than those previously established. 
This result considerably advances the understanding of the symmetric case 
\eqref{LV2S}, although a complete characterization of the propagation speed 
still remains an open problem. 

\begin{figure}[htbp]
\centering
\includegraphics[width=0.95\textwidth]{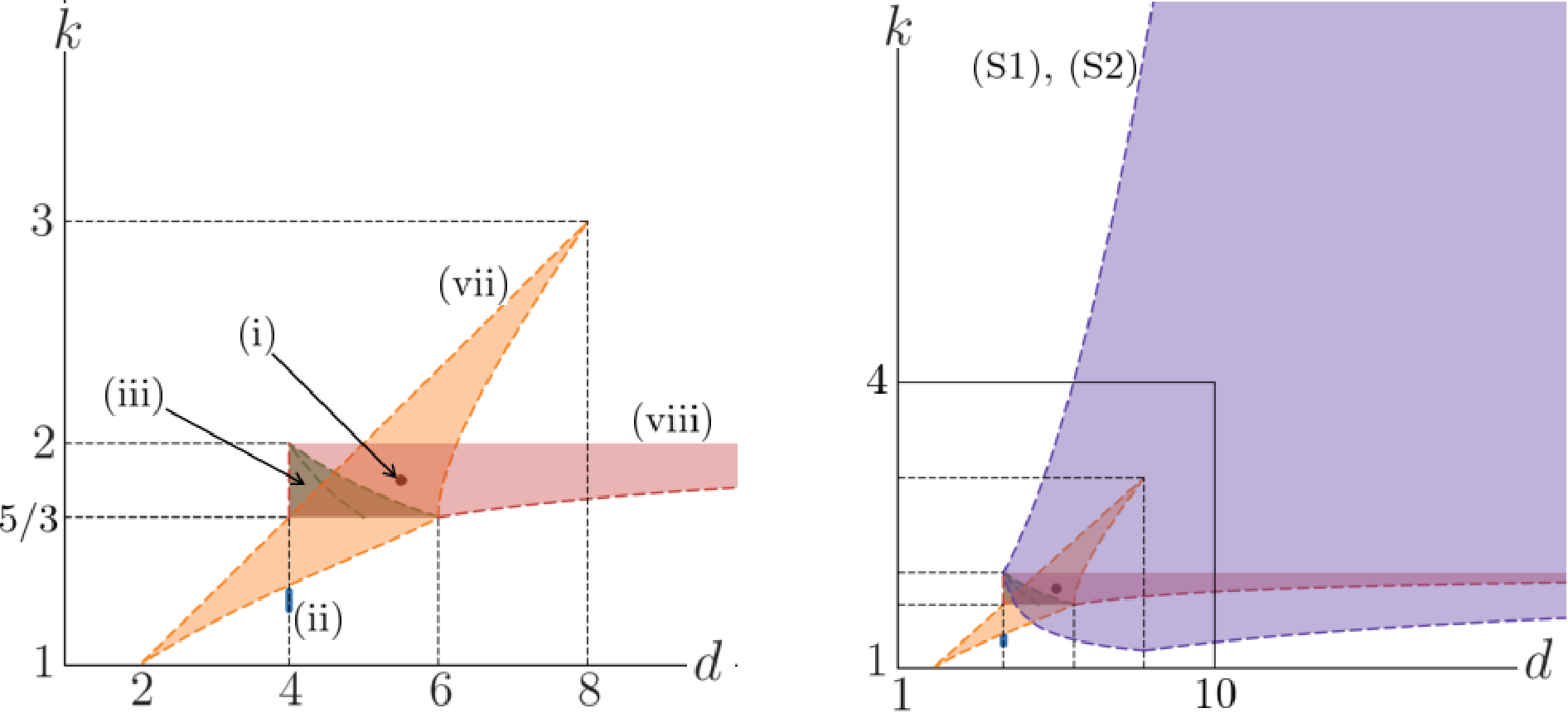} 
\caption{(Left) Parameter regions in the $(d,k)$-plane corresponding to 
negative speeds previously established. (Right) Additional regions (S1) and (S2) in Theorem \ref{thm:negative-sym} (highlighted in purple). 
The solid-line box $[1,10]\times[1,4]$ indicates the drawing area shown in the 
left figure.}
\label{fig:sym1}
\end{figure}
%%%%%%%%%%%%%%%%%%%%%%%%%%%%%%%%%%%%%%%%%%%%%%%%%%%%%%%%%%%%%%%%%%%%%%%%%%%%
\section{Nearly degenerate case}\label{S:degenerate}
In this section, we derive sufficient conditions for negative propagation 
speed in the case where the diffusion ratio $d$ is small. 

According to the result of Alzahrani, Davidson and Dodds 
\cite[Theorem 23 and Remark 24]{ADD10}, for sufficiently small $d>0$, 
the propagation speed $c$ is negative if $k_1>k_2^{\,2}$, whereas $c$ is 
positive if $k_1<k_2^{\,2}$. 
  
For the sake of clarity, we briefly recall how the threshold value $k_1=k_2^{\,2}$ 
arises. 
When $d=0$, a standing wave (that is, a traveling wave with propagation speed 
$0$) $(\phi,\psi)$ of \eqref{Co1} satisfies 
\begin{equation}\label{Co2}
\begin{cases}
\phi''+f(\phi,\psi)=0, \quad & x\in\BR, \\
g(\phi,\psi)=0, & x\in\BR, \\
(\phi,\psi)(-\infty)=(0,0),& (\phi,\psi)(\infty)=(1,1), 
\end{cases}
\end{equation}
where $f$ and $g$ are given in \eqref{def:fg}. From the definition of $g$, 
we seek a solution satisfying
\begin{equation}\label{psi-degenerate}
\psi=
\begin{cases}
k_2\phi, & x<0, \\
1, & x\geq0,
\end{cases}
\end{equation} 
together with the continuity condition $\phi(0)=1/k_2$. 

For $x<0$, the first equation of \eqref{Co2} reduces to 
\begin{equation}
\phi''=-f(\phi,k_2\phi)=(k_1-1)\phi-(k_1k_2-1)\phi^2,
\end{equation}
which yields  
\begin{equation}\label{phi-degenerate}
(\phi')^2=(k_1-1)\phi^2-\frac23(k_1k_2-1)\phi^3,
\end{equation}
under the conditions $\phi(-\infty)=0$ and $\phi'(-\infty)=0$.

For $x>0$, solving $\phi''=-f(\phi,1)=-\phi(1-\phi)$ subject to the conditions 
$\phi(\infty)=1$ and $\phi'(\infty)=0$, we obtain 
\begin{equation}\label{phi+degenerate}
(\phi')^2=-\phi^2+\frac23\phi^3+\frac13.
\end{equation}
By imposing the $C^1$-matching condition  
for $\phi$ at $x=0$, we have  
\begin{equation}
(k_1-1)k_2^{-2}-\frac23(k_1k_2-1)k_2^{-3}=-k_2^{-2}+\frac23 k_2^{-3}+\frac13,
\end{equation}
which yields $k_1=k_2^{\,2}$. Thus, \eqref{Co2} admits a standing wave 
precisely when $k_1=k_2^{\,2}$. 

Motivated by this observation, we construct a time-independent supersolution $(\phi_+,\psi_+)$ of \eqref{Co1} 
satisfying
\begin{equation}\label{upsi-small-d}
\psi_+=
\begin{cases}
k_2\phi_++\delta, & x<0, \\
1, & x\geq0,
\end{cases}
\end{equation} 
for small $d$, where $\delta\in(0,1)$ is a constant to be specified later. 
To ensure the continuity of $\psi_+$ at $x=0$, we impose the condition 
\begin{equation}\label{uphi0}
\phi_+(0)=\frac{1-\delta}{k_2}.
\end{equation}

Let $I(x)$ and $J(x)$ be the functions defined in \eqref{def:IJ}. Then, 
$(\phi_+,\psi_+)$ is a time-independent supersolution of \eqref{Co1} 
if $I\leq 0$ and $J\leq 0$ for all $x\not=0$ and 
\begin{equation}\label{jump}
\phi_+'(0-)\geq \phi_+'(0+), \quad \psi_+'(0-)\geq \psi_+'(0+).
\end{equation}

For $x<0$, we consider 
the equation $I(x)=\phi_+''+f(\phi_+,k_2\phi_++\delta)=0$, namely 
\begin{equation}\label{eq:uphi-}
\phi_+''-\{k_1(1-\delta)-1\}\phi_++(k_1k_2-1)(\phi_+)^2=0, \quad x<0. 
\end{equation}
When $\delta<\delta_1:=1-1/k_1$, \eqref{eq:uphi-} has a solution of the form 
\begin{equation}\label{sol:uphi-}
\phi_+(x)=\beta\mu(x)(1-\mu(x)), 
\end{equation}
where 
\begin{equation}\label{def:mu}
 \mu(x)=\frac{1}{1+e^{-\gamma(x-\xi)}} \enskip (\xi\in\BR), \quad 
 \gamma=\sqrt{k_1(1-\delta)-1}, \quad \beta=\frac{6\gamma^2}{k_1k_2-1}. 
\end{equation}
Let $m_0=\phi_+(0)/\beta=\mu(0)(1-\mu(0))$. Then, $m_0\leq 1/4$, and 
by \eqref{uphi0} and \eqref{def:mu}, 
\begin{equation}\label{m0}
 m_0=\dfrac{1-\delta}{k_2\beta}=\frac{(1-\delta)(k_1k_2-1)}{6k_2\{k_1(1-\delta)-1\}}>\frac16.
\end{equation}
Therefore, if we assume 
\begin{equation}\label{cond_delta1}
k_1>3-\frac{2}{k_2}, \quad 0<\delta< \delta_2:=1-\frac{3k_2}{k_1k_2+2} \ (<\delta_1),
\end{equation}
we can take $\xi>0$ satisfying $m_0=\mu(0)(1-\mu(0))\in(1/6,1/4)$.  
For such $\xi$, the solution $\phi_+$ is strictly monotone increasing in $x<0$ with $\phi_+(-\infty)=0$ and 
$\phi_+(0)=m_0\beta$.

In view of \eqref{uphi0}, \eqref{eq:uphi-} and \eqref{def:mu}, we see that 
\begin{align}
J=\frac{d}{r}k_2\phi_+''+g(\phi_+,k_2\phi_++\delta)=\dfrac{d}{r}\frac{6k_2\gamma^2}\beta\phi_+\left(\frac{\beta}{6}-\phi_+\right)-\delta k_2(m_0\beta-\phi_+). 
\end{align}
Since $J\leq0$ for $\beta/6\leq \phi_+\leq m_0\beta$, we only have to derive a condition $J\leq 0$ for $0<\phi_+<\beta/6$, or equivalently, 
\begin{equation}\label{cond:J<0}
\dfrac{d}{r}\leq H(\phi_+):=\frac{\beta \delta}{6\gamma^2}\frac{m_0\beta-\phi_+}{\phi_+\left(\beta/6-\phi_+\right)}, \quad 0<\phi_+<\frac{\beta}{6}.
\end{equation}
Since $H$ attains its minimum at $\phi_+=m_*\beta$, where 
\begin{equation}
m_*:=m_0-\sqrt{m_0\left(m_0-1/6\right)}<\frac16, 
\end{equation}
we obtain the following condition for $J\leq0$ in the case $x<0$: 
\begin{equation}\label{cond:J<0_2}
\dfrac{d}{r}\leq H_*:=H(m_*\beta)=\dfrac{\delta(m_0-m_*)}{\gamma^2m_*(1-6m_*)}.
%\frac{\delta}{6\gamma^2}\frac{\sqrt{m_0(m_0-1/6)}}{(m_0-\sqrt{m_0(m_0-1/6)})(1/6-m_0+\sqrt{m_0(m_0-1/6)})}. 
\end{equation}

Next, for $x>0$, we consider the equation $I(x)=\phi_+''+f(\phi_+,1)=0$, namely, 
\begin{equation}\label{eq:uphi+}
\phi_+''+\phi_+(1-\phi_+)=0, \quad x>0.
\end{equation}
This has a solution of the form
\begin{equation}\label{sol:uphi+}
\phi_+(x)=1-6\lambda(x)(1-\lambda(x)), 
\end{equation}
where $\lambda(x)=\left(1+e^{-(x-\eta)}\right)^{-1}$ $(\eta\in\BR)$. Then we can take $\eta<0$ such that 
the solution $\phi_+$ satisfies \eqref{uphi0} and is strictly monotone increasing in $x>0$ with 
$\phi_+(0)=m_0\beta$ and $\phi_+(\infty)=1$. On the other 
hand, $J(x)=0$ for all $x>0$ since $\psi_+\equiv1$.  

Finally, we will derive conditions for \eqref{jump}. We consider the former inequality since the latter obviously holds from \eqref{upsi-small-d}. 
By \eqref{eq:uphi-} and \eqref{eq:uphi+},  
\begin{align}
\phi_+'(0-)^2&=\{k_1(1-\delta)-1\}\phi_+(0)^2-\frac23(k_1k_2-1)\phi_+(0)^3, \\
\phi_+'(0+)^2&=-\phi_+(0)^2+\frac23 \phi_+(0)^3+\frac13.
\end{align}
Combining these with \eqref{uphi0}, we see that the inequality $\phi_+'(0-)\geq\phi_+'(0+)$ holds if 
\begin{equation}\label{cond_delta2}
k_1>k_2^{\,2}, \quad 0<\delta\leq \delta_3:=1-\left(k_1^{-1}k_2^{\,2}\right)^{1/3}.
\end{equation}
Since $k_2^{\,2}>3-2/k_2$ and $\delta_3<\delta_2$ for $k_1, k_2>1$, we conclude that $(\phi_+,\psi_+)$ 
defined by \eqref{sol:uphi-}, \eqref{sol:uphi+} and \eqref{upsi-small-d} becomes a time-independent 
supersolution of \eqref{Co1} if the conditions \eqref{cond_delta2} and \eqref{cond:J<0_2} are satisfied. 

Summarizing the above arguments, we obtain a sufficient condition 
for negative speed for nearly degenerate case: 
\begin{theorem}\label{thm:degenerate}
Suppose that $k_1>k_2^{\,2}$ and that
\begin{equation}\label{cond:small-d}
\dfrac{d}{r}<\frac{1-k_1^{-1/3}k_2^{2/3}}{\kappa(\kappa-1)(\kappa+1)^2}
\left(\sqrt{\kappa^2+\kappa+1}+1\right)^2, 
\end{equation}
where $\kappa:=(k_1k_2)^{1/3}>1$. Then the speed $c$ is negative. 
\end{theorem}
\begin{proof} First we note that 
\begin{equation}
H_*=\dfrac{\delta(m_0-m_*)}{\gamma^2m_*(1-6m_*)}
=\frac{6\delta}{k_1(1-\delta)-1}\left(\sqrt{m_0}+\sqrt{m_0-1/6}\right)^2.
\end{equation}
By \eqref{m0}, $m_0$ is monotone increasing in $\delta>0$ and hence so is $H_*$. 
Taking $\delta=\delta_3$ and letting $\kappa=(k_1k_2)^{1/3}>1$, we obtain
\begin{equation}
m_0=\dfrac{\kappa^2+\kappa+1}{6\kappa(\kappa+1)}
\end{equation}
and 
\begin{equation}
H_*=\frac{1-k_1^{-1/3}k_2^{2/3}}{\kappa(\kappa-1)(\kappa+1)^2}
\left(\sqrt{\kappa^2+\kappa+1}+1\right)^2.
\end{equation}
Therefore, by \eqref{cond:J<0_2} and the strictly monotone dependence of $c$ 
in $k_1$, 
the propagation speed $c$ is negative if \eqref{cond:small-d} holds.
\end{proof}

\begin{remark}\label{rem:ADD}
In \cite{ADD12}, Alzahrani, Davidson and Dodds numerically computed the curve 
in the $(k_1,d)$-plane, for fixed $r$ and $k_2$, on which the propagation speed 
$c$ is $0$. The curve clearly passes through the point $(k_2,r)$ and has 
the limiting 
points $(k_2^{\,2},0)$ and $(\sqrt{k_2},\infty)$ (\cite{ADD10}). 
They also conjectured that the curve is 
monotone, with $c<0$ to the right of the curve and $c>0$ to the left. Figure 
\ref{fig:ADD} (left) provides a schematic representation of their observations, 
shown on a double-logarithmic scale adapted from  
\cite[Figure 6]{ADD12} (see also \cite[Figure 3]{Gi2019}). 

Theorem \ref{thm:degenerate} together with Theorem \ref{thm:negative1} 
rigorously establishes a substantial portion of the negative-speed region suggested 
numerically; see Figure \ref{fig:ADD} (right). This provides a theoretical 
support for their conjecture, although its complete proof still remains open. 
 
\end{remark}

\begin{figure}[htbp]
\centering
\includegraphics[width=0.95\linewidth]{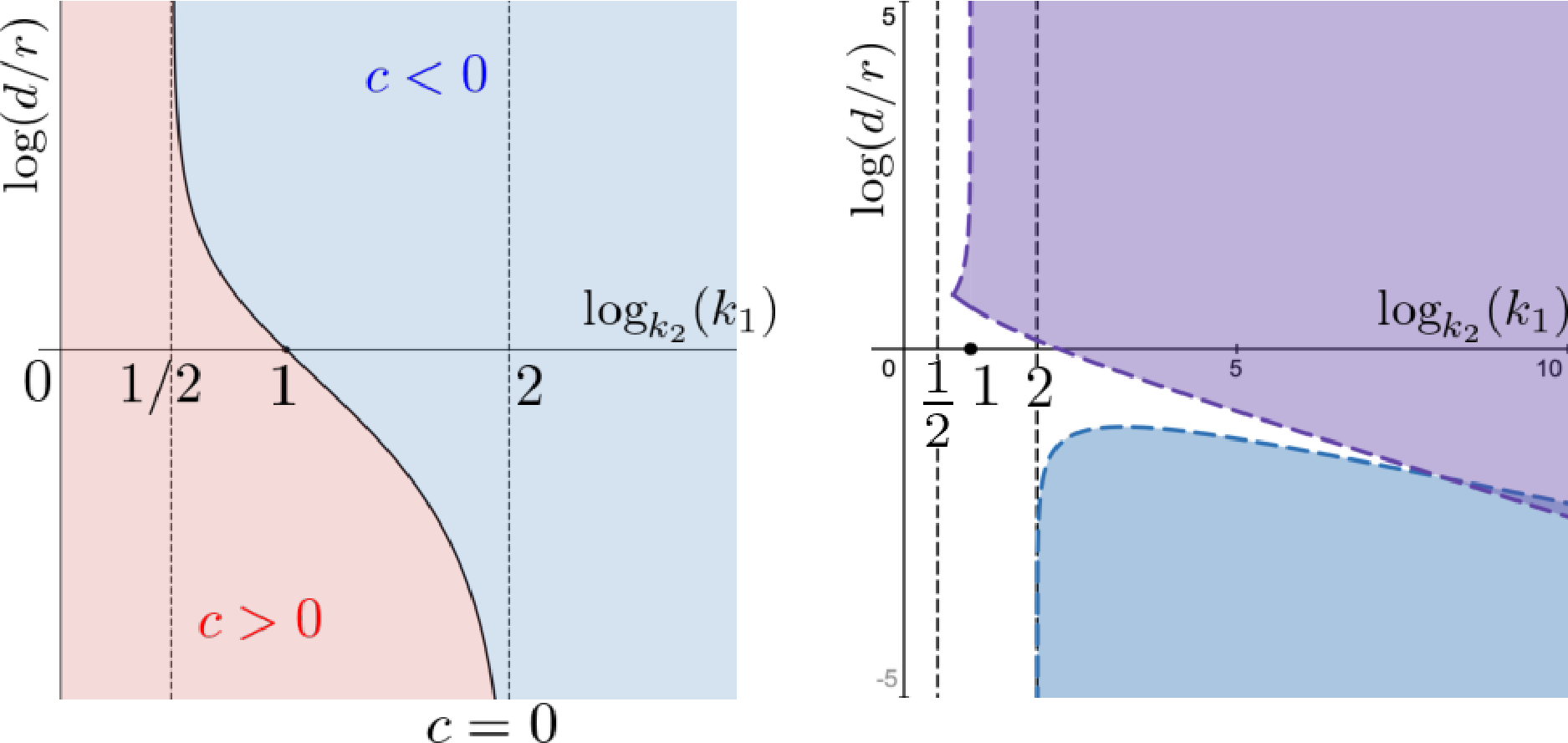} 
\caption{(Left) Schematic representation of the regions of negative speed (blue), positive speed (red), and zero speed (solid curve) numerically suggested in \cite{ADD12}.  (Right) Regions of negative speed for $r=1$ and $k_2=2$ 
established by Theorem \ref{thm:negative1} (purple) and Theorem 
\ref{thm:degenerate} (blue).}
\label{fig:ADD}
\end{figure}

%%%%%%%%%%%%%%%
\section{Determinacy of the speed sign under strongly asymmetric competition}\label{S:determinacy}
Combining Theorem \ref{thm:degenerate} with Theorem \ref{thm:negative1}, we 
establish the following result on the sign of the propagation speed, 
showing that when the interspecific competition coefficients differ significantly, the competitive outcome remains unaffected by adjusting diffusion rates 
or growth rates.
\begin{theorem}\label{thm:determinacy}
For any fixed $k_2>1$, there exist $k_1^*$ and $k_1^{**}$ with $1<k_1^*<k_1^{**}$ such that the speed $c$ is negative for all $d,r>0$ if $k_1\geq k_1^{**}$, while $c$ is positive for all $d,r>0$ if $1<k_1\leq k_1^*$. 
\end{theorem}
\begin{proof} First we consider the negative speed case. Let $k_2>1$ be fixed, 
and suppose  
$k_1\geq2$ and $k_1>k_2^{\,2}$. Then, since $k_1>m(k_2)=(\sqrt{24k_2+1}-3)/2$, the condition \eqref{neg1} 
yields that $c<0$ if 
\begin{equation}
\frac{d}{r}>\max\left\{\dfrac{m(k_2)^2}{k_1-1},\dfrac{2k_2m(k_2)}{2k_1-m(k_2)}\right\}=O(k_1^{-1}) 
\quad (k_1\to\infty).
\end{equation}
On the other hand, in view of \eqref{cond:small-d}, we see that $c<0$ if 
\begin{equation}
\dfrac{d}{r}<\frac{1-k_1^{-1/3}k_2^{2/3}}{\kappa(\kappa-1)(\kappa+1)^2}
\left(\sqrt{\kappa^2+\kappa+1}+1\right)^2=O(k_1^{-2/3}) \quad (k_1\to\infty),  
\end{equation}
where $\kappa=(k_1k_2)^{1/3}$. Therefore, we can find $k_1^{**}>1$ such that $c(d,r,k_1^{**},k_2)<0$ for all $d,r>0$. Since $c$ is strictly monotone decreasing in $k_1$, we obtain the 
assertion for negative speed. 

Next we consider the positive speed case. Let $k_2>1$ be fixed, and suppose 
$1<k_1<2$ and $m(k_1)\leq k_2$. Then, \eqref{neg1} and \eqref{formula} yield that $c>0$ if 
\begin{equation}\label{cond:lower}
\frac{r}{d}
>\max\left\{\frac{6k_2^{\,2}}{(k_2-1)^2(k_2+4)},\frac{4}{k_2-1}\right\}(k_1-1).
%\to0 \quad (\mbox{as} \ k_1\to1).
\end{equation}
On the other hand, we use \eqref{cond:small-d} and \eqref{formula} to conclude that $c>0$ if 
\begin{equation}\label{cond:upper}
\dfrac{r}{d}<\frac{1-k_2^{-1/3}k_1^{2/3}}{\kappa(\kappa-1)(\kappa+1)^2}
\left(\sqrt{\kappa^2+\kappa+1}+1\right)^2, 
\end{equation}
where $\kappa=(k_1k_2)^{1/3}$. Since the right-hand side of \eqref{cond:lower} approaches $0$ as 
$k_1\to1$ and since that of \eqref{cond:upper} is bounded away from $0$ as $k_1\to1$, there exists 
$k_1^*>1$ such that $c(d,r,k_1^*,k_2)>0$ for all $d,r>0$. Hence the strict 
monotonicity of $c$ in $k_1$ proves the assertion for positive speed. 
\end{proof}

\begin{remark}
As stated in Remark \ref{rem:ADD}, numerical observations in \cite{ADD12} conjecture that $k_1^*=\sqrt{k_2}$ and $k_1^{**}=k_2^{\,2}$. 
However, a rigorous proof has not yet been established. 
\end{remark}

%\backmatter


%%\bibliography{2comp_speed_SN1}% common bib file

%%\end{document}

\end{document}